\documentclass[12pt,a4paper]{article}
\usepackage[latin1]{inputenc}
\usepackage{amsmath}
\usepackage{amsfonts}
\usepackage{amssymb}
\usepackage{graphicx}
\usepackage{amsthm}
\usepackage{hyperref}
\usepackage{indentfirst}
\usepackage{url}
\usepackage{mathrsfs}

\usepackage{comment}

\newtheorem{theorem}{Theorem}[section]
\newtheorem{proposition}[theorem]{Proposition}
\newtheorem{corollary}[theorem]{Corollary}
\newtheorem{lemma}[theorem]{Lemma}

\theoremstyle{definition}

\newtheorem{definition}[theorem]{Definition}
\newtheorem{example}[theorem]{Example}
\newtheorem{question}[theorem]{Question}

\newtheorem{remark}[theorem]{Remark}

\DeclareMathOperator{\ad}{ad}

\title{Patterns on the numerical duplication by their admissibility degree}

\author{Alessio Borzì
		\thanks{Dipartimento di Matematica e Informatica, Università degli studi di Catania.} \thanks{Scuola Superiore di Catania.}}
	
\begin{document}
	
	\maketitle
	
	\begin{abstract}
		We develop the theory of patterns on numerical semigroups in terms of the admissibility degree. We prove that the Arf pattern induces every strongly admissible pattern, and determine all patterns equivalent to the Arf pattern. We study patterns on the numerical duplication $S \Join^d E$ when $d \gg0$. We also provide a definition of patterns on rings.
	\end{abstract}
	
	\section*{Introduction}
	
	A numerical semigroup $S$ is an additive submonoid of $\mathbb{N}$ with finite complement in $\mathbb{N}$. The set of values of a Noetherian, one-dimensional, analytically irreducible, local, domain is a numerical semigroup, therefore the study of numerical semigroups is related to the study of this class of rings. In \cite{lipman1971stable}, Lipman introduces and motivates the study of Arf rings, which constitute an important class of rings for the classification problem of singular curve branches. A good reference for the study of Arf rings in the analytically irreducible case is \cite{barucci1997maximality}. The value semigroup of an Arf ring is an Arf numerical semigroup. We say that a numerical semigroup $S$ is Arf if for every $x,y,z \in S$ with $x \geq y \geq z$ we have $x+y-z \in S$. There are several works in the literature about Arf numerical semigroups, see for instance \cite{rosales2004arf}, \cite{garcia2017parametrizing}. Note that Arf semigroups are related to the polynomial $x+y-z$. In \cite{bras2006patterns}, Bras-Am\'{o}ros and Garc\'{i}a-S\'{a}nchez generalize the definition of Arf semigroup to any linear homogeneous polynomial, introducing the theory of patterns on numerical semigroups \cite{stokes2016patterns}, \cite{bras2013nonhomogeneous}, \cite{stokes2014linear}, \cite{sun2017generalizing}.
	
	In this manner, Arf numerical semigroups are the semigroups that admit the \emph{Arf pattern} $x+y-z$. In addition, Arf numerical semigroups can be characterized in terms of their additive behaviour (see for instance \cite{bras2003improvements}, \cite{bras2009numerical}). Therefore, one can translate similar characterizations for certain classes of patterns.
	
	Given a numerical semigroup $S$ we can consider the quotient of $S$ by a positive integer $d \in \mathbb{N}$
	\[ \frac{S}{d} = \{ x \in \mathbb{N}: dx \in S \}. \]
	In \cite{d2013numerical}, D'Anna and Strazzanti define a semigroup construction, called the numerical duplication, that is, in a certain sense, the reverse operation of the quotient by $2$. If $A \subseteq \mathbb{N}$, the set of doubles is denoted by $2 \cdot A = \{ 2a: a \in A \}$ (note that $2 \cdot A \neq 2A = A+A$). Given a numerical semigroup $S$, a \emph{semigroup ideal} of $S$ is a subset $E \subseteq S$ such that $E+S \subseteq E$. If $d \in S$ is an odd integer, the numerical duplication of $S$ with respect to the semigroup ideal $E$ and $d$ is
	\[ S \Join^d E = 2 \cdot S \cup (2 \cdot E + d). \]
	The numerical duplication can be seen as the value semigroup of a quadratic quotient of the Rees algebra, see for instance \cite{barucci2015family}, \cite{borzi2018characterization}.
	This construction generalizes Nagata's idealization and the amalgamated duplication (see \cite{d2007amalgamated}), and it is one of the main tools used in \cite{oneto2017one} to give a negative answer to a problem of Rossi \cite{rossi2011hilbert}.
	
	In \cite{borzi2018characterization} it was characterized when the numerical duplication $S \Join^d E$ is Arf. The characterization is given in terms of the multiplicity sequence of the Arf semigroup $S$. A natural question is how this characterization can be generalized to any pattern. This paper deals with this question.
	
	In particular, in Section \ref{Section 2} and \ref{Section 3} we develop the theory of patterns on numerical semigroups in terms of the admissibility degree, generalizing some results of \cite{bras2006patterns} proved for Boolean patterns. Further, we prove that the Arf pattern induces every strongly admissible pattern and we determine the family of patterns equivalent to the Arf pattern. In Section \ref{Section 4} we characterize when the numerical duplication $S \Join^d E$ admits a monic pattern for $d \gg 0$ and give some examples of the general case. In Section \ref{Section 5} we give some observations and trace possible future work about pattern on rings. \\
	
	Several computations are performed by using the GAP system \cite{gap2015gap} and, in
	particular, the NumericalSgps package \cite{delgadonumericalsgps}.

	\section{Preliminaries}
	
	Let $S$ be a numerical semigroup, the \emph{multiplicity} of $S$ is the integer $\operatorname{m}(S) = \min (S \setminus \{0\})$, the \emph{conductor} of $S$ is $\operatorname{c}(S) = \min\{ x \in \mathbb{N}: x+\mathbb{N} \subseteq S \}$. If $E \subseteq S$ is a semigroup ideal of $S$, set $\operatorname{c}(E) = \min\{ x \in \mathbb{N}: x+\mathbb{N} \subseteq E \}$. Note that, if $d \in S$ is an odd integer, from \cite[Proposition 2.1]{d2013numerical} the conductor of the numerical duplication is $c(S \Join^d E)=2\operatorname{c}(E)+d-1$.
	
	A \emph{pattern} $p(x_1,\dots,x_n)$ of length $n$ is a linear homogeneous polynomial in $n$ variables with non-zero integer coefficients. The pattern of length zero is the zero polynomial $p=0$. A numerical semigroup $S$ \emph{admits} a pattern $p$ if for every $s_1,\dots,s_n \in S$ with $s_1 \geq \dots \geq s_n$ we have $p(s_1,\dots,s_n) \in S$. The family of all numerical semigroups admitting $p$ is denoted by $\mathscr{S}(p)$. Given two patterns $p_1,p_2$, we say that $p_1$ \emph{induces} $p_2$ if $\mathscr{S}(p_1) \subseteq \mathscr{S}(p_2)$; we say that $p_1$ and $p_2$ are \emph{equivalent} if they induce each other, or equivalently $\mathscr{S}(p_1) = \mathscr{S}(p_2)$. 
	Let $p$ be a pattern of length $n$, set
	\[ p(x_1,\dots,x_n) = \sum_{i=1}^n a_i x_i,  \]
	and $b_i = \sum_{j \leq i}a_j$, we will keep this notation throughout. Note that we can write
	\[
	\begin{split}
	p(x_1,\dots,x_n) & = a_1x_1+\dots+a_nx_n = \\
	& = b_1(x_1-x_2)+\dots+b_{n-1}(x_{n-1}-x_n)+b_nx_n,
	\end{split}
	\]
	we will use frequently this decomposition in the sequel. The pattern $p$ is \emph{admissible} if $\mathscr{S}(p) \neq \emptyset$, that is, $p$ is admitted by some numerical semigroup. Set
	\[ p' = \begin{cases}
		p-x_1 & \text{if } a_1 > 1 \\
		p(0,x_1,\dots,x_{n-1}) & \text{if } a_1 = 1,
	\end{cases} \]
	and define recursively $p^{(0)} = p$ and $p^{(i)} = (p^{(i-1)})'$ for $i \in \mathbb{N} \setminus \{0\}$.	The \emph{admissibility degree} of $p$, denoted by $\ad(p)$, is the least integer $k$ such that $p^{(k)}$ is not admissible, if such integer exists, otherwise is $\infty$. If $p'$ is admissible, $p$ is \emph{strongly admissible}. With this definitions, $p$ is admissible if $\ad(p) \geq 1$, strongly admissibile if $\ad(p) \geq 2$.
	
	\begin{proposition}\cite[Theorem 12]{bras2006patterns}\label{admissible characteriazation}
		For a pattern $p$ the following conditions are equivalent
		\begin{enumerate}
			\item $p$ is admissible,
			\item $\mathbb{N}$ admits $p$,
			\item $b_i \geq 0$ for all $i \in \{1,\dots,n\}$.
		\end{enumerate}
	\end{proposition}
	
	\begin{corollary}\label{admissibility 1 exist bi = 0}
		If $p$ has admissibility degree $1$, then there exists $i \in \{ 1,\dots,n \}$ such that $b_i = 0$.
	\end{corollary}
	\begin{proof}
		By hypothesis $p'$ is not admissibile, then from Proposition \ref{admissible characteriazation} there exists $i$ such that $(a_1-1) + \sum_{j=2}^i a_j = -1 \Rightarrow b_i = \sum_{j=1}^i a_j = 0$. 
	\end{proof}
	
	The \emph{trivializing pattern} is $x_1-x_2$, note that $\mathscr{S}(x_1-x_2) = \{ \mathbb{N} \}$, so from Proposition \ref{admissible characteriazation} it induces every admissibile pattern, in other words it induces every pattern $p$ with $\ad(p) \geq 1$. The \emph{Arf pattern} is $x_1+x_2-x_3$, it is equivalent to $2x_1 - x_2$ (see \cite[Example 5]{bras2006patterns}). The family $\mathscr{S}(x_1+x_2-x_3)$ is the family of Arf numerical semigroups. More in general, the \emph{subtraction pattern} of degree $k$ is the pattern $x_1+x_2+\dots+x_k-x_{k+1}$. So the trivializing pattern and the Arf pattern are the subtraction patterns of degree $1$ and $2$. Note that the admissibility degree of a subtraction pattern is equal to its degree.
	
	Following \cite[Chapter 6]{rosales2009numerical}, a \emph{Frobenius variety} is a nonempty family $\mathscr{F}$ of numerical semigroups such that
	\begin{enumerate}
		\item $S,T \in \mathscr{F} \Rightarrow S \cap T \in \mathscr{F}$,
		\item $S \in \mathscr{F} \setminus \{ \mathbb{N} \} \Rightarrow S \cup \{ F(S) \} \in \mathscr{F}$.
	\end{enumerate}
	
	\begin{proposition}\cite[Proposition 7.17]{rosales2009numerical}\label{frobenius varieties}
		If $p$ is a strongly admissible pattern, then $\mathscr{S}(p)$ is a Frobenius variety.
	\end{proposition}
	
	Given a Frobenius variety $\mathscr{F}$, it is possible to define the closure of a numerical semigroup $S$ as the smallest (with respect to set inclusion) numerical semigroup in $\mathscr{F}$ that contains $S$. From this idea, we can define the notion of system of generators with respect to the variety. In addition, we can construct a tree of all numerical semigroups in $\mathscr{F}$ rooted in $\mathbb{N}$ and such that $T$ is a son of $S$ if and only if $T = S \cup \{F(S)\}$.
	
	From Proposition \ref{frobenius varieties}, these definitions generalize many notions given in \cite{bras2006patterns}, for instance $p$-closure or $p$-system of generators.
	
	\section{Patterns and their admissibility degree}\label{Section 2}
	
	
	In \cite{stokes2016patterns} and \cite{sun2017generalizing} it was noted that a pattern $p$ is strongly admissibile (i.e. $\ad(p) \geq 2)$ if and only if $b_i \geq 1$ for all $i \in \{1,\dots,n\}$. Of course if $b_i \geq k$ for all $i \in \{1,\dots,n\}$ then $\ad(p) \geq k+1$.
	
	\begin{proposition}\label{admissibility min k,i}
		If a pattern $p$ has admissibility degree at least $k+1$, then $b_i \geq \nolinebreak \min\{i,k\}$ for all $i \in \{ 1,\dots,n \}$.
	\end{proposition}
	\begin{proof}
		Let $a_i'$ be the coefficients of $p'$ and $b_i' = \sum_{j\leq i} a_j'$. We proceed by induction on $k$. The base case follows from Proposition \ref{admissible characteriazation}. For the inductive step, firstly we assume that $p$ is monic. For all $i \in \{1,\dots,n-1\}$ we have $b_{i+1} = b_i'+1$, then
		\[ \ad(p) \geq k+1 \Rightarrow \ad(p') \geq k \Rightarrow b_i' \geq \min\{i,k-1\} \Rightarrow b_{i+1} \geq \min\{i+1,k\}, \]
		in addition $b_1 = 1 \geq \min\{1,k\}$. On the other hand, if $p$ is not monic, for all $i \in \{1,\dots,n\}$ we have $b_i = b_i'+1$, then
		\begin{gather*}
			\ad(p) \geq k+1 \Rightarrow \ad(p') \geq k \Rightarrow \\
			\Rightarrow b_i' \geq \min\{i,k-1\} \Rightarrow b_{i} \geq \min\{i+1,k\} \geq \min\{i,k\}.\qedhere
		\end{gather*}
	\end{proof}
	
	\begin{example}
		Proposition \ref{admissibility min k,i} cannot be inverted. For instance consider the pattern $p = x_1+3x_2-x_3$, then $b_i \geq \min\{i,k\}$ for all $k \in \mathbb{N}$, but $p$ has admissibility degree $4$.
	\end{example}

	The next result generalizes \cite[Lemma 42]{bras2006patterns} and the proof is similar.
	
	\begin{lemma}\label{Head Center Tail}
		An admissible pattern $p$ with finite admissibility degree can be written uniquely as
		\[ p(x_1,\dots,x_n) = H_p(x_1,\dots,x_h) + C_p(x_{h},\dots,x_t) + T_p(x_{t+1},\dots,x_n), \]
		where either $H_p = 0$ or all the coefficients of $H_p$ are positive and their sum is equal to $\ad(p)-1$, $C_p$ is admissible and the sum of all its coefficients is zero, $\ad(T_p) > 1$.
	\end{lemma}
	\begin{proof}
		Set $\ad(p) = k+1$, then $p$ can be written uniquely as the sum
		\[ p(x_1,\dots,x_n) = H_p(x_1,\dots,x_h) + p^{(k)}(x_h,\dots,x_n) \]
		where $H_p$ is a pattern with positive coefficients and their sum is equal to $k = \ad(p)-1$, and $p^{(k)}$ is admissible with $\ad(p^{(k)}) = 1$. If $a_i'$ are the coefficients of $p^{(k)}$, by Corollary \ref{admissibility 1 exist bi = 0} there exists an integer $i$ such that $\sum_{j=h}^i a_j' = 0$, set $t$ to be the largest of such integers. Set
		\[ C_p(x_h,\dots,x_t) = \sum_{i=h}^t a_i'x_i, \quad T_p(x_{t+1},\dots,x_n) = \sum_{i=t+1}^{n} a_i'x_i. \]
		By the choice of $t$ it follows $\sum_{i=t+1}^m a_i' = \sum_{i=h}^m a_i' > 0$ for all $m \in \{t+1, \dots, n\}$, hence $\ad(T_p) > 1$.
	\end{proof}
	\noindent
	If the pattern $p$ has admissibilty degree $\infty$, we set $H_p = p$ and $C_p = T_p = 0$. Therefore, we can write every pattern as
	\begin{equation}\label{eq:H+C+T}
		p(x_1,\dots,x_n) = H_p(x_1,\dots,x_h) + C_p(x_{h},\dots,x_t) + T_p(x_{t+1},\dots,x_n)
	\end{equation}
	we will keep this notation throughout.
	
	\begin{definition}
		Let $p$ be a pattern. With the notation of Lemma \ref{Head Center Tail} we call $H_p$ the \emph{head}, $C_p$ the \emph{center} and $T_p$ the \emph{tail} of $p$. The decomposition (\ref{eq:H+C+T}) is the \emph{standard decomposition} of $p$.
	\end{definition}
	
	\begin{example}
		Let $p = x_1 + 3 x_2 + x_3 -2 x_4 + x_5 + x_6$, the admissibility degree of $p$ is $4$, the standard decomposition of $p$ is
		\begin{align*}
			H_p(x_1,x_2) &= x_1 + 2 x_2, \\
			C_p(x_2,x_3,x_4) &= x_2 + x_3 - 2 x_4, \\
			T_p(x_5,x_6) &= x_5+x_6.
		\end{align*}
	\end{example}
	
	\begin{corollary}
		Any non-zero strongly admissible pattern $p$ can be decomposed into the sum
		\[ p = p_1 + q_1 + p_2 + q_2 + \dots + p_m + q_m, \]
		where the coefficients of the pattern $p_i$ are positive, the pattern $q_i$ is admissible and the sum of its coefficients is zero, for all $i \in \{ 1,\dots,m \}$.
	\end{corollary}
	\begin{proof}
		It follows by recursively applying Lemma \ref{Head Center Tail} on the tail of $p$.
	\end{proof}
	
	\begin{remark}\label{centered patterns}
		Note that the head of every pattern of admissibility degree $1$ is zero. Further, if $p$ is an admissible pattern in which the sum of all coefficients is zero (i.e. $b_n=0$), the tail of $p$ is zero. In addition, by Proposition \ref{admissibility min k,i}, the admissibility degree of $p$ is $1$, so the head of $p$ is also zero, consequently $p$ is equal to its center. Therefore, an admissible pattern is equal to its center if and only if the sum of all its coefficients is equal to zero.
	\end{remark}
	
	The next result follows a similar idea of \cite[Proposition 2.4]{sun2017generalizing}. 
	
	\begin{proposition}\label{admissibility 1}
		Let $p$ be an admissible pattern such that the sum of its coefficients is zero. A numerical semigroup $S$ admits $p$ if and only if the monoid generated by the integers $b_1,\dots,b_n$ is a subset of $S$.
	\end{proposition}
	\begin{proof}
		\emph{Necessity}. Let $i \in \{1,\dots,n\}$ and $\lambda \in \mathbb{N}$ such that $\lambda,\lambda+1 \in S$. Then
		\[ p(\underbrace{\lambda+1,\dots,\lambda+1}_{i },\lambda,\dots,\lambda) = \sum_{j=1}^n a_j \lambda + \sum_{j=1}^i a_j = b_n \lambda + b_i = b_i \in S. \]
		\emph{Sufficiency}. It is enough to write
		\[
			\begin{split}
				p(x_1,\dots,x_n) & = a_1x_1+\dots+a_nx_n = \\
				& = b_1(x_1-x_2)+\dots+b_{n-1}(x_{n-1}-x_n)+b_nx_n. \qedhere
			\end{split}
		\]
	\end{proof}
	
	\begin{proposition}\label{admissibility 1 p iff Cp and Tp}
		If $p$ has admissibility degree $1$, then a numerical semigroup $S$ admits $p$ if and only if it admits $C_p$ and $T_p$.
	\end{proposition}
	\begin{proof}
		Sufficiency follows from $p = C_p+T_p$. For the necessity it is enough to write
		\begin{gather*}
			p(x_1,\dots,x_t,0,\dots,0) = C_p(x_1,\dots,x_t), \\
			p(x_{t+1},\dots,x_{t+1},x_{t+2},\dots,x_n) = T_p(x_{t+1},\dots,x_n),
		\end{gather*}
		where $t$ is the same index used in the proof of Lemma \ref{Head Center Tail}.
	\end{proof}
	
	\begin{corollary}\label{structure admissibility 1}
		If $p$ has admissibility degree $1$, then a numerical semigroup $S$ admits $p$ if and only if $S$ admits $T_p$ and contains the monoid generated by $b_1,\dots,b_t$.
	\end{corollary}
	
	By iterating on the tail, the previous Corollary \ref{structure admissibility 1} with \cite[Lemma 14]{bras2006patterns} gives us an algorithm to determine if a numerical semigroup admits an admissible pattern. Further, the previous result allows us to extend Proposition \ref{frobenius varieties} to (not necessarily strongly) admissible patterns.

	\begin{proposition}\label{structure admissibility 2 monic}
		If $p$ is monic and has admissibility degree $2$ with
		\[ p(x_1,\dots,x_n) = x_1 + C_p(x_2,\dots,x_t)+T_p(x_{t+1},\dots,x_n), \]
		then $S$ admits $p$ if and only if it admits $p_i(x_1,x_2,x_3) = x_1+(b_i-1)(x_2-x_3)$ for all $i \in \{2,\dots,n\}$, and $x_1+T_p$.
	\end{proposition}
	\begin{proof}
		First, write
		\[ p(x_1,\dots,x_n) = x_1 + \sum_{i=2}^t (b_i-1)(x_i-x_{i+1})+T_p(x_{t+1},\dots,x_n). \]
		\emph{Necessity}. Let $i \in \{2,\dots,n\}$, we have
		\begin{gather*}
			p(x_1,\underbrace{x_2,\dots,x_2}_{i-1 },\underbrace{x_3,\dots,x_3}_{t-i },0,\dots,0) = x_1+(b_i-1)(x_2-x_3), \\
			p(\underbrace{x_1,\dots,x_1}_{t },x_{t+1},\dots,x_n) = x_1 + T_p(x_{t+1},\dots,x_n).
		\end{gather*}
		\emph{Sufficiency}. Let $\lambda_1,\dots,\lambda_n \in S$ with $\lambda_1\geq \dots \lambda_n$. We can write
		\[ p(\lambda_1,\dots,\lambda_t,0,\dots,0) = \lambda_1 + \sum_{i=2}^t (b_i-1)(\lambda_i-\lambda_{i+1}). \]
		By hypothesis $\lambda_1 + (b_2-1)(\lambda_2-\lambda_3) \in S$ and it is greater than $\lambda_1$. Thus also $\Big( \lambda_1 + (b_2-1)(\lambda_2-\lambda_3) \Big) + (b_3-1)(\lambda_3-\lambda_4) \in S$. By iterating this process we obtain $p(\lambda_1,\dots,\lambda_t,0,\dots,0) \in S$ and it is greater than $\lambda_1$. Finally, since $S$ admits $x_1+T_p$, we have
		\[ p(\lambda_1,\dots,\lambda_n) = p(\lambda_1,\dots,\lambda_t,0,\dots,0)+T_p(\lambda_{t+1},\dots,\lambda_n) \in S. \qedhere \]
	\end{proof}
	
	\section{Patterns equivalent to the Arf pattern}\label{Section 3}
	
	The next result is a straightforward generalization of \cite[Proposition 34]{bras2006patterns}.
	
	\begin{lemma}\label{conductor / multiplicity}
		A numerical semigroup $S$ admits every pattern of admissibility degree greater or equal than $\lceil \frac{\operatorname{c}(S)}{\operatorname{m}(S)} \rceil +1$.
	\end{lemma}
	\begin{proof}
		Write
		\[ p(x_1,\dots,x_n) = H_p(x_1,\dots,x_h) + C_p(x_{h+1},\dots,x_t) + T_p(x_{t+1},\dots,x_n), \]
		and recall that the coefficients of $H_p$ are positive and their sum is equal to $\ad(p)-1 \geq \lceil \frac{\operatorname{c}(S)}{\operatorname{m}(S)} \rceil$.
		Let $s_1,\dots,s_n \in S$ with $s_1 \geq \dots \geq s_n$. If $s_{h+1} < \operatorname{m}(S)$, then $s_{h+1} = s_{h+2} = \dots = s_n = 0$ and
		\[ p(s_1,\dots,s_n) = \sum_{i=1}^h a_is_i \in S. \]
		On the other hand, if $s_{h+1} \geq \operatorname{m}(S)$, then $s_1 \geq \dots s_h \geq \operatorname{m}(S)$, therefore
		\[ \begin{split}
			p(s_1,\dots,s_n) \geq H_p(s_1,\dots,s_h)  & \geq (\ad(p)-1)\operatorname{m}(S) \geq \\
			& \geq \left\lceil \frac{\operatorname{c}(S)}{\operatorname{m}(S)} \right\rceil \operatorname{m}(S) \geq \operatorname{c}(S). \qedhere
		\end{split} \]
	\end{proof}
	
	
	\begin{proposition}
		If $p$ has admissibility degree $k$, then there exists a numerical semigroup $S$ that admits every pattern of admissibility degree $k+1$ but it does not admit $p$.
	\end{proposition}
	\begin{proof}
		If $k=0$ take $S = \mathbb{N}$. Assume $k \geq 1$. The sum of the coefficients of $C_p$ si zero, therefore we can write
		\[ p(x_1,\dots,x_n) = H_p(x_1,\dots,x_h) + \sum_{i=h+1}^{t}c_i(x_i-x_{i+1}) + T_p(x_{t+1},\dots,x_n) \]
		for some $c_i \in \mathbb{N}$. Note that there exists $r \in \{ h+1,\dots,t \}$ such that $c_r > 0$. Now let $q \in \mathbb{N}$ such that $q > c_r+k-1$. Set $S = \langle q , q+1 \rangle \cup (kq + \mathbb{N})$, then
		\[ p(\underbrace{q+1,\dots,q+1}_{r },\underbrace{q,\dots,q}_{t-r },0,\dots,0) = (k-1)(q+1)+c_r = \lambda, \]
		with $(k-1)q + k-1 < \lambda < kq$, therefore $\lambda \notin S$, so $S$ does not admit $p$. Nonetheless, since $\operatorname{c}(S) = kq = k\operatorname{m}(S)$, from the preceding lemma $S$ admits every pattern of admissibility degree $k+1$.
	\end{proof}
	
	\begin{corollary}\label{equivalent same degree}
		Let $p$ and $q$ be two patterns.
		\begin{enumerate}
			\item If $p$ induces $q$, then $\ad(p) \leq \ad(q)$.
			\item If $p$ and $q$ are equivalent, then $\ad(p) = \ad(q)$.
		\end{enumerate}
	\end{corollary}
	
	\begin{lemma}\label{Arf induces x1+n(x2-x3)}
		The Arf pattern induces the pattern $x_1 + n(x_2-x_3)$ for every $n \in \mathbb{N}$.
	\end{lemma}
	\begin{proof}
		We prove this by induction on $n$. The case $n=0$ is the pattern $x_1$, the case $n=1$ is the Arf pattern itself. For the inductive step, suppose that the Arf pattern induces $x_1+n(x_2-x_3)$, then it is enough to write
		\[ x_1+(n+1)(x_2-x_3) = \Big( x_1+n(x_2-x_3) \Big) + x_2-x_3. \qedhere \]
	\end{proof}
	
	Recall that a pattern $p$ is strongly admissible if and only if it has admissibility degree at least $2$.
	
	\begin{proposition}\label{Arf induces ad geq 2}
		The Arf pattern induces every strongly admissible pattern.
	\end{proposition}
	\begin{proof}
		Let $p$ be a strongly admissible pattern, so $\ad(p) \geq 2$. We proceed by induction on the number of variables $n$ of the pattern $p$. If $n=1$ then $p$ is equivalent to the zero pattern, so the Arf pattern induces $p$. Now, for the inductive step, suppose that the Arf pattern induces every pattern of admissibility degree at least $2$ with at most $n-1$ variables. Since $p'$ induces $p$, it is enough to prove that the Arf pattern induces every pattern of admissibility degree $2$ with $n$ variables. So assume $\ad(p) = 2$. Suppose that $S$ admits the Arf pattern. Let $s_1,\dots,s_n \in S$ with $s_1 \geq \dots \geq s_n$. From Lemma \ref{Head Center Tail} we have
		\[
			p(s_1,\dots,s_n) = s_1 + \sum_{i=1}^{t-1} (b_i-1)(s_i-s_{i+1})+T_p(s_{t+1},\dots,s_n),
		\]
		note that $b_t - 1 = 0$ and $t>1$. From Lemma \ref{Arf induces x1+n(x2-x3)} the Arf pattern induces the pattern $x_1+(b_1-1)(x_2-x_3)$, so $s_1' = s_1 + (b_1-1)(s_1-s_2) \in S$ with $s_1' \geq s_1$. Similarly, since the Arf pattern induces the pattern $x_1 + (b_2-1)(x_2-x_3)$, then
		\[ s_2' = s_1' + (b_2-1)(s_2-s_3) = s_1 + (b_1-1)(s_1-s_2) + (b_2-1)(s_2-s_3) \in S, \]
		with $s_2' \geq s_1$. Iterating this process we obtain that
		\[ s = s_1 + \sum_{i=1}^{t-1} (b_i-1)(s_i-s_{i+1}) \in S. \]
		Since $t > 1$, the number of variables of the pattern $x_1+T_p$ is less than $n$. By the inductive hypothesis, the Arf pattern induces $x_1+T_p(x_{t+1},\dots,x_n)$, so
		\[ p(s_1,\dots,s_n) = s + T_p(s_{t+1},\dots,s_n) \in S. \qedhere \]
	\end{proof}
	
	What we have so far is that for $k=1,2$, the subtraction pattern of degree $k$ induces all patterns of admissibility degree at least $k$. As \cite[Example 50]{bras2006patterns} shows, this cannot be extended to $k \geq 3$.
	
	\begin{theorem}\label{Arf equivalent to bi=2}
		A pattern $p = \sum_{i=1}^n a_ix_i$ is equivalent to the Arf pattern if and only if it has admissibility degree $2$ and there exists $i \in \{1,\dots,n\}$ such that $b_i = \sum_{j=1}^i a_j = 2$.
	\end{theorem}
	\begin{proof}
		From Corollary \ref{equivalent same degree}, we can assume that $\ad(p) = 2$. Now from Proposition \ref{Arf induces ad geq 2}, the Arf pattern induces $p$. If there exists $i$ such that $b_i = 2$, then
		\begin{gather*}
			p(x_1,\dots,x_n) = x_1 + \sum_{i=1}^t (b_i-1)(x_i-x_{i+1})+T_p(x_{t+1},\dots,x_n) \Rightarrow \\
			\Rightarrow p(\underbrace{x_1,\dots,x_1}_{i },\underbrace{x_2,\dots,x_2}_{t-i },0,\dots,0) = x_1 + (b_i-1)(x_1-x_2) = 2x_1-x_2.
		\end{gather*}
		Therefore $p$ induces the pattern $2x_1-x_2$ which is equivalent to the Arf pattern. On the other hand, suppose that $b_i \neq 2$ for all $i \in \{ 1,\dots,n \}$. Then, from Proposition \ref{admissibility min k,i}, either $b_i = 1$ or $b_i \geq 3$. Let $q >1$ and $S = \{ q,q+1,q+3,\rightarrow \}$. From Lemma \ref{conductor / multiplicity}, $S$ admits every pattern of admissibility degree greater or equal than $3$. In particular, $S$ admits $x_1+T_p$. Now let $s_1,\dots,s_n \in S$ with $s_1 \geq \dots \geq s_n$. If for every $i \in \{ 1,\dots,t \}$ either $b_i = 1$ or $s_i = s_{i+1}$, then
		\[ p(s_1,\dots,s_n) = s_1 + T_p(s_{t+1},\dots,s_n) \in S. \]
		Otherwise, there exists $i \in \{1,\dots t\}$ such that $s_i > s_{i+1}$ and $b_i \geq 3$, then $s_1 \geq s_i > s_{i+1} \geq q \Rightarrow s_1 \geq q+1$, and
		\[ p(s_1,\dots,s_n) \geq s_1 + (b_i-1)(s_{i+1}-s_i) \geq q+3 = \operatorname{c}(S).  \]
		Clearly, $S$ is not Arf since $2(q+1)-q = q+2 \notin S$, therefore $p$ is not equivalent to the Arf pattern.
	\end{proof}
	
	Note that Corollary \ref{structure admissibility 1} and Theorem \ref{Arf equivalent to bi=2}, generalize and provide another proof of \cite[Proposition 48]{bras2006patterns}, since if $p$ is a Boolean pattern of admissiblity degree $k$, then $b_k = k$.
	
	\section{Patterns on the numerical duplication}\label{Section 4}
	
	In this section $S$ will be a numerical semigroup, $E$ will be an ideal of $S$, $d \in S$ will be an odd integer and $p = \sum_{i=1}^n a_ix_i$ will be an admissible pattern. We say that the numerical duplication $S \Join^d E$ \emph{admits $p$ eventually with respect to $d$} if there exists $d' \in \mathbb{N}$ such that $S \Join^d E$ admits $p$ for all $d \geq d'$.
	
	\begin{proposition}
		If $S$ admits $p$ then also $\frac{S}{k}$ admits $p$ for every $k \geq 1$.
	\end{proposition}
	\begin{proof}
		If $\lambda_1 \geq \dots \geq \lambda_n$ are elements of $\frac{S}{k}$, then $k\lambda_1 \geq \dots \geq k\lambda_n$ are in $S$. Therefore
		\[ p(k\lambda_1,\dots,k\lambda_n) = k p(\lambda_1,\dots,\lambda_n) \in S \Rightarrow p(\lambda_1,\dots,\lambda_n) \in \frac{S}{k}. \qedhere \]
	\end{proof}
	
	For the next result, recall that $\frac{S \Join^d E}{2} = S$.
	
	\begin{corollary}
		If $S \Join^d E$ admits p then $S$ admits $p$.
	\end{corollary}
	
	Throughout we will assume that $S$ admits the pattern $p$.
	Note that if $p$ has admissibility degree $2$, then by applying Corollary \ref{admissibility 1 exist bi = 0} to the center of $p$, we obtain that the set $B = \{ i: b_i-1=0 \}$ is nonempty.
	
	
	\begin{proposition}\label{num dup ad 2 1}
		Suppose that $p$ has admissibility degree $2$ and set
		\[ B = \{ i : b_i-1 = 0 \}, \quad r = \min B, \quad t = \max B. \]
		If $S \Join^d E$ admits $p$ eventually with respect to $d$, then
		\begin{enumerate}
			\item for every $1 \leq i \leq t$, $\lfloor \frac{b_i}{2} \rfloor \in E-E$;
			\item for every $r \leq i \leq t$, if $b_i$ is even then $b_i/2 \geq \operatorname{c}(E)-\min(E)$.
		\end{enumerate}
	\end{proposition}
	\begin{proof}
		From Lemma \ref{Head Center Tail}, we can write $p$ in the following manner
		\[ p = x_1 + \sum_{i=1}^{t-1}(b_i-1)(x_i-x_{i+1}) + T_p(x_{t+1},\dots,x_n). \]
		Now, assume $d \geq 2\operatorname{c}(S)-2\min(E)+1$, then we have that
		\[ (2\min(E)+d-1) +2\cdot\mathbb{N} \subseteq 2\operatorname{c}(S) + 2 \cdot \mathbb{N} \subseteq 2 \cdot S \subseteq S \Join^d E. \]
		Let $i \in \{1,\dots,t-1 \}$, $e \in E$ and fix $\lambda = 2e + d-1$. By the assumption on $d$ we have that $\lambda \in S \Join^d E$. If $b_i$ is odd, it follows that
		\[ p(\underbrace{\underbrace{\lambda+1,\dots,\lambda+1}_{i },\lambda,\dots,\lambda}_{t },0,\dots,0) = \lambda+1 + (b_i-1) = \]
		\[ = 2e+d+b_i-1 = 2(e+(b_i-1)/2) + d \in S \Join^d E, \]
		hence $e+(b_i-1)/2 \in E$, so by the arbitrary choice of $e \in E$ we have $(b_i-1)/2 \in E-E$. On the other hand, if $b_i$ is even, then
		\[ p(\underbrace{\underbrace{\lambda+2,\dots,\lambda+2}_{i },\lambda+1,\dots,\lambda+1}_{r },0,\dots,0) = \lambda+2 + (b_i-1) = \]
		\[ = 2e+d+b_i = 2(e+b_i/2) + d \in S \Join^d E \]
		hence $e+b_i/2 \in E$, so as before $b_i/2 \in E-E$. This proves that $\lfloor \frac{b_i}{2} \rfloor \in E-E$. Now let $i \in \{r,\dots,t\}$ such that $b_i$ is even and set $\lambda = 2\min(E) + d + 1$. Let $x \in \mathbb{N}$ and set $\mu = \lambda+2x$. Again by the assumption on $d$ we have that $\mu,\lambda \in S \Join^d E$. Thus
		\[ p(
		\underbrace{
			\underbrace{
				\underbrace{
					\mu,\dots,\mu}_{r }
				,\lambda\dots,\lambda}_{i }
			,\lambda-1,\dots,\lambda-1}_{t }
		,0,\dots,0) = \mu + (b_i-1) = \]
		\[ = \lambda+2x+b_i-1 = 2(\min(E)+b_i/2+x) + d \in S \Join^d E, \]
		hence $\min(E)+b_i/2+x \in E$. By the arbitrary choice of $x$ we have that $\min(E) + b_i/2 \geq \operatorname{c}(E) \Rightarrow b_i/2 \geq \operatorname{c}(E)-\min(E)$.
	\end{proof}
	
	\begin{proposition}\label{admissibility 2}
		If $p$ has admissibility degree $2$ and is monic, then $S \Join^d E$ admits $p$ eventually with respect to $d$ if and only if
		\begin{enumerate}
			\item for every $i \in \{1,\dots,t\}$
			\begin{itemize}
				\item if $b_i$ is odd then $(b_i-1)/2 \in E-E$;
				\item if $b_i$ is even then $b_i/2 \geq \operatorname{c}(E)-\min(E)$.
			\end{itemize}
			\item $S \Join^d E$ admits $x_1+T_p$.
		\end{enumerate}
	\end{proposition}
	\begin{proof}
		\emph{Necessity}. The first condition follows from Proposition \ref{num dup ad 2 1} since we have $b_1 - 1 = 0$. Further, if we take $x_2=x_3=\dots=x_t$, then
		\[ p(x_1,x_2\dots,x_2,x_{t+1},\dots,x_n) = x_1 + T_p(x_{t+1},\dots,x_n). \]
		\emph{Sufficiency}. From Proposition \ref{structure admissibility 2 monic} it is enough to show that $S \Join^d E$ admits $p_i(x_1,x_2,x_3) = x_1 + (b_i-1)(x_2-x_3)$ for all $i \in \{ 2,\dots,n \}$. Let $i \in \{2,\dots,n\}$ and $\lambda_1,\lambda_2,\lambda_3 \in S \Join^d E$ with $\lambda_1 \geq \lambda_2 \geq \lambda_3$. If $\lambda_2 = \lambda_3$, then $p_i(\lambda_1,\lambda_2,\lambda_3) = \lambda_1 \in S\Join^d E$, so we can assume $\lambda_2 > \lambda_3$. Since $S$ admits $p$, it admits also $p_i$, so if $\lambda_1 < 2\min(E)+d$ then $\lambda_1,\lambda_2,\lambda_3 \in 2\cdot S$ and $p_i(\lambda_1,\lambda_2,\lambda_3) \in 2 \cdot S \subseteq S \Join^d E$. Now assume that $\lambda_1 \geq 2\min(E)+d$. If $b_i$ is even then, $b_i \geq 2(\operatorname{c}(E)-\min(E))$ and we have
		\[
		\begin{split}
		p_i(\lambda_1,\lambda_2,\lambda_3) & = \lambda_1 + (b_i-1)(\lambda_2-\lambda_3) \geq \lambda_1+b_i-1 \geq \\
		& \geq 2 \min(E) + d + 2(\operatorname{c}(E)-\min(E))-1 = \\
		& = 2\operatorname{c}(E)+d-1 = c(S\Join^d E),
		\end{split}
		\]
		therefore $p_i(\lambda_1,\lambda_2,\lambda_3) \in S \Join^d E$. On the other hand, if $b_i$ is odd, then $\mu = (b_i-1)(\lambda_2-\lambda_3) \in 2(E-E)$ since $E-E$ is a semigroup. Now if $\lambda_1$ is even, then $\lambda_1+\mu$ is also even, so for $d \gg 0$ we have $\lambda_1+\mu \in 2 \cdot S \subseteq S \Join^d E$. If $\lambda_1 = 2e+d \in 2 \cdot E+d$, then $\lambda_1+\mu = 2(e+\mu/2)+d \in 2 \cdot E+d \subseteq S \Join^d E$ since $\mu \in 2(E-E)$.
	\end{proof}
	
	\begin{proposition}\label{admissibility 3 not monic}
		If $p$ has admissibility degree at least $3$ and it is not monic (i.e. $a_1 \geq 2$), then $S \Join^d E$ admits $p$ eventually with respect to $d$.
	\end{proposition}
	\begin{proof}
		Let $\lambda_1,\dots,\lambda_n \in S \Join^d E$ with $\lambda_1 \geq \dots \geq \lambda_n$. Since $S$ admits $p$, if $\lambda_1 < 2\min(E)+d$ then $\lambda_i \in 2\cdot S$ for all $i \in \{1,\dots,n\}$ and we have $p(\lambda_1,\dots,\lambda_n) \in 2 \cdot S \subseteq S \Join^d E$. Now assume that $\lambda_1 \geq 2\min(E)+d$. Note that, since $p$ has admissibility degree at least $3$, $p''$ is admissible, so $p''(\lambda_1,\dots,\lambda_n) \geq 0$. Now if we take $d \geq 2\operatorname{c}(E) - 4\min(E)$, then
		\[ 
		\begin{split}
			p(\lambda_1,\dots,\lambda_n) = 2\lambda_1 + p''(\lambda_1,\dots,\lambda_n) & \geq 4\min(E)+2d \geq \\
			& \geq 2\operatorname{c}(E)+d \geq c(S \Join^d E),
		\end{split}
		\]
		hence $p(\lambda_1,\dots,\lambda_n) \in S \Join^d E$.
	\end{proof}
	
	
	
	\begin{proposition}\label{admissibility 3}
		If $p$ is monic with admissibility degree at least 3, then $p'(S) \subseteq E-E$ if and only if $S \Join^d E$ admits $p$ eventually with respect to $d$.
	\end{proposition}
	\begin{proof}
		\emph{Necessity}. Let $\lambda_1,\dots,\lambda_n \in S \Join^d E$ with $\lambda_1 \geq \dots \geq \lambda_n $. First assume that $\lambda_2 < 2\min(E)+d$, so $\lambda_i = 2 s_i$ with $s_i \in S$ for all $i \geq 2$. Now if $\lambda_1 \in 2\cdot S$, then $p(\lambda_1,\dots,\lambda_n) \in 2 \cdot S \subseteq S \Join^d E$. Otherwise, if $\lambda_1 = 2e+d \in 2\cdot E + d$, then fix $g = p'(s_2,\dots,s_n) \in p'(S) \subseteq E-E$, we have $g+e \in E$, hence
		\[
		\begin{split}
			p(\lambda_1,\dots,\lambda_n) & = 2e+d + p'(2s_2,\dots,2s_n) =\\
			& =  2e+d + 2 p'(s_2,\dots,s_n) = \\
			& = 2(e+g)+d \in 2\cdot E + d \subseteq S \Join^d E.
		\end{split}
		\]
		On the other hand, if $\lambda_2 \geq 2\min(E)+d$, take $d \geq 2\operatorname{c}(E)-4\min(E)$. Since $p$ has admissibility degree at least 3, $p''$ is admissible, so $p''(\lambda_2,\dots,\lambda_n) \geq 0$, then
		\[
		\begin{split}
		p(\lambda_1,\dots,\lambda_n) = \lambda_1 + \lambda_2 + p''(\lambda_2,\dots,\lambda_n) & \geq 4\min(E)+2d \geq \\
		& \geq 2\operatorname{c}(E)+d \geq c(S \Join^d E),
		\end{split}
		\]
		hence $p(\lambda_1,\dots,\lambda_n) \in S \Join^d E$. \\
		\emph{Sufficiency}. Let $g = p'(s_1,\dots,s_{n-1}) \in p'(S)$, with $s_1, \dots, s_{n-1} \in S$ and $s_1 \geq \dots \geq s_{n-1}$. Let $e \in E$, it is enough to prove that $g+e \in E$. If $2s_1 < 2\min(E)+d \leq 2e+d$, it follows that
		\[
		\begin{split}
			p(2e+d,2s_1,\dots,2s_{n-1}) & = 2e+d+p'(2s_1,\dots,2s_{n-1}) = \\
			& = 2e+d+2p'(s_1,\dots,s_{n-1}) = \\
			& = 2(e+g)+d \in S \Join^d E \Rightarrow e+g \in E.
		\end{split}
		\]
		On the other hand, if $2s_1 \geq 2\min(E)+d$, take $d \geq 2\operatorname{c}(E) -2\min(E)$, then
		\[
		\begin{split}
			2g & = p'(2s_1,\dots,2s_{n-1}) = \\
			& = 2s_1 + p''(2s_1,\dots,2s_{n-1}) \geq 2s_1 \geq 2 \min(E)+d \geq 2\operatorname{c}(E),
		\end{split}
		\]
		hence $g \geq \operatorname{c}(E) \Rightarrow g \in E-E$.
	\end{proof}

	Assembling Corollary \ref{structure admissibility 1}, Proposition \ref{admissibility 2} and Proposition \ref{admissibility 3} and iterating these results on $H_p+T_p$, we are able to characterize when the numerical duplication $S \Join^d E$ admits a monic pattern $p$ for $d \gg 0$.
	
	\begin{theorem}
		Let $p$ be a monic pattern, written as
		\[ p(x_1,\dots,x_n) = H_p(x_1,\dots,x_h)+C_p(x_{h+1},\dots,x_t) + T_p(x_{t+1},\dots,x_n). \]
		Then $S \Join^d E$ admits $p$ eventually with respect to $d$ if and only if one of the following cases occurs:
		\begin{enumerate}
			\item $\ad(p) = 1$, $S \Join^d E = \mathbb{N}$.
			
			\item $\ad(p) = 2$, for every $i \in \{1,\dots,t\}$
				\begin{itemize}
					\item if $b_i$ is odd then $(b_i-1)/2 \in E-E$;
					\item if $b_i$ is even then $b_i/2 \geq \operatorname{c}(E)-\min(E)$;
				\end{itemize}
			and $S \Join^d E$ admits $x_1+T_p$.
			
			\item $\ad(p) \geq 3$ and $p'(S) \subseteq E-E$.
		\end{enumerate}
	\end{theorem}
	
	From Proposition \ref{num dup ad 2 1}, Proposition \ref{admissibility 3 not monic} and Corollary \ref{structure admissibility 1}, in order to extend the previous theorem to not monic patterns, we would need just a sufficient condition in the case $\ad(p)=2$.
	
	In the general case, that is when $d$ can be small, we can extend the characterization of \cite[Theorem 2.4]{borzi2018characterization} by combining it with Theorem \ref{Arf equivalent to bi=2}. Nonetheless, as the following examples show, it seems complicated to find a sort of characterization for a generic pattern.
	
	\begin{example}
		The following tables show for which values of $d$ the numerical duplication $S \Join^d E$ admits $p$.
		
		\begin{center}
			\begin{tabular}{c c}
				$\begin{array}[t]{c}
					S = \langle 3,19,20 \rangle \\
					E = 3+S \\
					p(x_1,x_2) = 3x_1-x_2. \\[10pt]
					\begin{tabular}{|c | c|}
						\hline
						$d$ & admits $p$ \\ \hline \hline
						$3$ & $\checkmark$ \\ \hline
						$9$ & $\checkmark$ \\ \hline
						$15$ & $\checkmark$ \\ \hline
						$19$ &  \\ \hline
						$21$ & $\checkmark$ \\ \hline
						$23$ &  \\ \hline
						$25$ &  \\ \hline
						$27$ & $\checkmark$ \\ \hline
						$29$ & $\checkmark$ \\ \hline
					\end{tabular}
				\end{array}$
				&
				$\begin{array}[t]{c}
					S = \langle 5,8,19,22 \rangle \\
					E = 5+S \\
					p(x_1,x_2,x_3) = 4x_1-x_2-x_3. \\[10pt]
					\begin{tabular}{|c | c|}
						\hline
						$d$ & admits $p$ \\ \hline \hline
						$5$ & \\ \hline
						$13$ & $\checkmark$ \\ \hline
						$15$ & \\ \hline
						$19$ & $\checkmark$ \\ \hline
						$21$ & \\ \hline
						$23$ & \\ \hline
						$25$ & $\checkmark$ \\ \hline
						$27$ & $\checkmark$ \\ \hline
						$29$ & $\checkmark$ \\ \hline
					\end{tabular}
				\end{array}$
			\end{tabular}
		\end{center}
	\end{example}
	
	\section{Patterns on rings}\label{Section 5}
	
	In this section, $(R,\mathfrak{m})$ will be a one-dimensional, Noetherian, Cohen-Macaulay, local ring, $\overline{R}$ will be the integral closure of $R$ in its total ring of fractions $Q(R)$. An ideal $I$ of $R$ is \emph{open} if it contains a regular element. We will assume that the residue field $k = R/\mathfrak{m}$ is infinite. From \cite[Proposition 1.18, pag 74]{kiyek2012resolution}, the last condition assures that every open ideal $I$ has an \emph{$I$-transversal element}, namely an element $x \in I$ such that $xI^n = I^{n+1}$ for $n \gg 0$. On $R$ we define the following \emph{preorder} (namely a reflexive and transitive relation): let $x,y \in R$, then $x \leq_R y$ if $y/x \in \overline{R}$. Let $p$ be the pattern
	\[ p(x_1,\dots,x_n) = \sum_{i=1}^n a_ix_i. \]
	
	\begin{definition}
		The ring $R$ \emph{admits} the pattern $p$ if for every $y_1,\dots,y_n \in R$ with $y_1 \geq_R \dots \geq_R y_n$, we have
		\[ y_1^{a_1}y_2^{a_2}\dots y_n^{a_n} \in R. \]
	\end{definition}
	
	With this definition, when the relation $\leq_R$ is a total order, $R$ is an Arf ring if and only if it admits the Arf pattern. Note that $R$ admits the trivializing pattern if and only if $R = \overline{R}$.
	
	\begin{remark}\label{algebraic n(x1-x2)}
		The ring $R$ admits the pattern $n(x_1-x_2)$, with $n \in \mathbb{N}$, if and only if for every $z \in \overline{R}$ it results $z^n \in R$. In fact, for every $z \in \overline{R} \subseteq Q(R)$, there exist $x,y \in R$ such that $z = y/x$, and by definition $y \geq_R x$.
	\end{remark}
	
	From the previous remark we can determine when a ring $R$ admits a pattern of admissibility degree $1$ applying, mutatis mutandis, Corollary \ref{structure admissibility 1}.
	
	\begin{corollary}
		The ring $R$ admits a pattern $p$ of admissibility degree $1$ if and only if it admits $T_p$ and for every $z \in \overline{R}$ it results $z^{b_i} \in R$ for all $i \in \{1,\dots n\}$.
	\end{corollary}

	Similarly, if $p$ is monic and $\ad(p)=2$, we can apply, mutatis mutandis, Proposition \ref{structure admissibility 2 monic}
	
	Now we make additional assumptions on $R$. Following \cite{barucci1997maximality}, let $V$ be a discrete valuation domain with valuation $v: V \rightarrow \mathbb{N}$, and let $\mathcal{V}$ be the set of all subrings $R$ of $V$ such that $R$ is a local, Noetherian, one-dimensional, analytically irreducible, residually rational, domain and its integral closure $\overline{R}$ is equal to $V$. Set $\mathscr{V}(p)$ be the family of rings in $\mathcal{V}$ that admit the pattern $p$. If $p_1$ and $p_2$ are two patterns, then it is clear that if $\mathscr{V}(p_1) \subseteq \mathscr{V}(p_2)$ then $\mathscr{S}(p_1) \subseteq \mathscr{S}(p_2)$, i.e. $p_1$ induces $p_2$. A question naturally arise.
	
	\begin{question}
		Is the implication $\mathscr{S}(p_1) \subseteq \mathscr{S}(p_2) \Rightarrow \mathscr{V}(p_1) \subseteq \mathscr{V}(p_2)$ true?
	\end{question}

	Now fix $R \in \mathcal{V}$, note that, since $\overline{R}$ is a valuation ring, $\leq_R$ is a total preorder. Further, $x \leq_R y$ if and only if $v(x) \leq v(y)$. In this setting, the integral closure of an ideal $I$ of $R$ is
	\[ \overline{I} = I\overline{R} \cap I = \{ x \in R: v(x) \geq \min v(I) \}, \]
	(see \cite[Proposition 1.6.1, Proposition 6.8.1]{huneke2006integral}). It is not difficult to prove (see for instance \cite[Theorem 2.2]{lipman1971stable} or \cite[Theorem II.2.13]{barucci1997maximality}) that $R$ is an Arf ring if and only if $I^2 = xI$ for every integrally closed ideal $I \subseteq R$ and some $x \in I$ of minimum value. Actually, the inclusion $xI \subseteq I^2$ is always true, so what we actually prove is that $I^2 \subseteq xI$. We can generalize this idea to any subtraction pattern of degree $k$.
	
	\begin{proposition}
		The ring $R$ admits the subtraction pattern of degree $k$ if and only if $I^k \subseteq xI$ for every integrally closed ideal $I \subseteq R$ and some $x \in I$ of minimum value.
	\end{proposition}
	\begin{proof}
		\emph{Necessity}. Let $i_1,i_2,\dots,i_k \in I$, since $\leq_R$ is a total preorder, we can assume that $i_1 \geq_R \dots \geq_R i_k$. If $x \in I$ is an element of minimum value, then $i_k \geq_R x$. By hypothesis $i_1i_2 \dots i_k x^{-1} \in I$, hence $i_1i_2 \dots i_k \in xI$. \\
		\emph{Sufficiency}. Let $y_1,\dots,y_{k+1} \in R$ with $y_1 \geq_R \dots \geq_R y_{k+1}$. Set $I$ to be the integral closure of $Ry_{k+1}$. Since $v(y_1) \geq v(y_2) \geq \dots \geq v(y_{k+1})$ and $I$ is integrally closed, then $y_i \in I$ for all $i \in \{ 1,\dots,k \}$. By hypothesis $y_1 \dots y_k \in I^k \subseteq y_{k+1}I$, then $y_1\dots y_k y_{k+1}^{-1} \in I \subseteq R$.
	\end{proof}
	
	In \cite[Theorem II.2.13]{barucci1997maximality} it was proved that $R$ is Arf if and only if $v(R)$ is Arf and the multiplicity sequence of $R$ and $v(R)$ coincides. 
	
	\begin{question}
		For an arbitrary pattern $p$ are there any characterization similar to the previous one?
	\end{question}
	
	\noindent
	\textbf{Acknowledgments}. I would like to thank Marco D'Anna for his constant support, Maria Bras-Am\'{o}ros for useful conversations and email exchanges, and Nicola Maugeri for indicating some good references.

	\bibliographystyle{abbrv}
	\bibliography{Reference.bib}

\begin{thebibliography}{10}

\bibitem{barucci2015family}
V.~Barucci, M.~D'Anna, and F.~Strazzanti.
\newblock {A family of quotients of the Rees algebra}.
\newblock {\em Communications in Algebra}, 43(1):130--142, 2015.

\bibitem{barucci1997maximality}
V.~Barucci, D.~E. Dobbs, and M.~Fontana.
\newblock {\em Maximality properties in numerical semigroups and applications
  to one-dimensional analytically irreducible local domains}, volume 598.
\newblock American Mathematical Soc., 1997.

\bibitem{borzi2018characterization}
A.~Borz{\`\i}.
\newblock {A characterization of the Arf property for quadratic quotients of
  the Rees algebra}.
\newblock {\em arXiv preprint arXiv:1806.04448}, 2018.

\bibitem{bras2003improvements}
M.~Bras-Amor{\'o}s.
\newblock Improvements to evaluation codes and new characterizations of arf
  semigroups.
\newblock In {\em International Symposium on Applied Algebra, Algebraic
  Algorithms, and Error-Correcting Codes}, pages 204--215. Springer, 2003.

\bibitem{bras2009numerical}
M.~Bras-Amor{\'o}s.
\newblock On numerical semigroups and the redundancy of improved codes
  correcting generic errors.
\newblock {\em Designs, Codes and Cryptography}, 53(2):111, 2009.

\bibitem{bras2006patterns}
M.~Bras-Amor{\'o}s and P.~A. Garc{\'i}a-S{\'a}nchez.
\newblock Patterns on numerical semigroups.
\newblock {\em Linear algebra and its applications}, 414(2-3):652--669, 2006.

\bibitem{bras2013nonhomogeneous}
M.~Bras-Amor{\'o}s, P.~A. Garc{\'\i}a-S{\'a}nchez, and A.~Vico-Oton.
\newblock Nonhomogeneous patterns on numerical semigroups.
\newblock {\em International Journal of Algebra and Computation},
  23(06):1469--1483, 2013.

\bibitem{d2007amalgamated}
M.~D'Anna and M.~Fontana.
\newblock An amalgamated duplication of a ring along an ideal: the basic
  properties.
\newblock {\em Journal of Algebra and its Applications}, 6(03):443--459, 2007.

\bibitem{d2013numerical}
M.~D'Anna and F.~Strazzanti.
\newblock The numerical duplication of a numerical semigroup.
\newblock In {\em Semigroup forum}, volume~87, pages 149--160. Springer, 2013.

\bibitem{delgadonumericalsgps}
M.~Delgado, P.~A. Garc{\'\i}a-S{\'a}nchez, and J.~Morais.
\newblock {NumericalSgps. A package for numerical semigroups, Version 1.0.1}.
\newblock 2015.
\newblock \url{http://www.gap-system.org/Packages/numericalsgps.html}.

\bibitem{garcia2017parametrizing}
P.~A. Garc{\'i}a-S{\'a}nchez, B.~A. Heredia, H.~Karaka{\c{s}}, and J.~C.
  Rosales.
\newblock {Parametrizing Arf numerical semigroups}.
\newblock {\em Journal of Algebra and Its Applications}, 16(11):1750209, 2017.

\bibitem{huneke2006integral}
C.~Huneke and I.~Swanson.
\newblock {\em Integral closure of ideals, rings, and modules}, volume~13.
\newblock Cambridge University Press, 2006.

\bibitem{kiyek2012resolution}
K.~Kiyek and J.~L. Vicente.
\newblock {\em Resolution of curve and surface singularities in characteristic
  zero}, volume~4.
\newblock Springer Science \& Business Media, 2012.

\bibitem{lipman1971stable}
J.~Lipman.
\newblock {Stable ideals and Arf rings}.
\newblock {\em American Journal of Mathematics}, 93(3):649--685, 1971.

\bibitem{oneto2017one}
A.~Oneto, F.~Strazzanti, and G.~Tamone.
\newblock {One-dimensional Gorenstein local rings with decreasing Hilbert
  function}.
\newblock {\em Journal of Algebra}, 489:91--114, 2017.

\bibitem{gap2015gap}
\relax The GAP~Group.
\newblock Gap--groups, algorithms, and programming, version 4.7.9;.
\newblock 2015.
\newblock \url{http://www.gap-system.org}.

\bibitem{rosales2009numerical}
J.~C. Rosales and P.~A. Garc{\'\i}a-S{\'a}nchez.
\newblock {\em Numerical semigroups}, volume~20.
\newblock Springer Science \& Business Media, 2009.

\bibitem{rosales2004arf}
J.~C. Rosales, P.~A. Garc{\'i}a-S{\'a}nchez, J.~I. Garc{\'i}a-Garc{\'i}a, and
  M.~B. Branco.
\newblock Arf numerical semigroups.
\newblock {\em Journal of Algebra}, 276(1):3--12, 2004.

\bibitem{rossi2011hilbert}
M.~E. Rossi.
\newblock {Hilbert functions of Cohen-Macaulay local rings}.
\newblock {\em Commutative Algebra and Its Connections to Geometry:
  Pan-American Advanced Studies Institute, August 3-14, 2009, Universidade
  Federal de Pernambuco, Olinda, Brazil}, 555:173, 2011.

\bibitem{stokes2016patterns}
K.~Stokes.
\newblock Patterns of ideals of numerical semigroups.
\newblock In {\em Semigroup Forum}, volume~93, pages 180--200. Springer, 2016.

\bibitem{stokes2014linear}
K.~Stokes and M.~Bras-Amor{\'o}s.
\newblock Linear, non-homogeneous, symmetric patterns and prime power
  generators in numerical semigroups associated to combinatorial
  configurations.
\newblock In {\em Semigroup Forum}, volume~88, pages 11--20. Springer, 2014.

\bibitem{sun2017generalizing}
G.~Sun and Z.~Zhao.
\newblock Generalizing strong admissibility of patterns of numerical
  semigroups.
\newblock {\em International Journal of Algebra and Computation},
  27(01):107--119, 2017.

\end{thebibliography}
	
\end{document}